\documentclass[10pt]{article}
\usepackage{geometry}
\usepackage{color}
\usepackage{float}
\usepackage{textcomp}
\usepackage{amsthm}
\usepackage{amsmath}
\usepackage{amssymb}
\usepackage[T1]{fontenc}
\usepackage{lipsum}
\usepackage{multirow}
\usepackage{multicol}      % para ajustar alinhamentos de colunas

\usepackage{changes}

\usepackage{mathtools}

\newtheorem{theorem}{Theorem}[section]
\newtheorem{lemma}[theorem]{Lemma}

\newtheorem{proposition}[theorem]{Proposition}
\newtheorem{remark}[theorem]{Remark}

\newcommand{\R}{\mathbb{R}}

\newcommand{\inner}[2]{\langle{#1},{#2}\rangle}
\newcommand{\Inner}[2]{\left\langle{#1},{#2}\right\rangle}
\newcommand{\norm}[1]{\|#1\|}
\newcommand{\Norm}[1]{\left\|#1\right\|}

\newcommand{\tos}{\rightrightarrows} 
\newcommand{\Y}{\mathcal{Y}}
\newcommand{\X}{\mathcal{X}}
\newcommand{\Z}{\mathcal{Z}}

\newcommand{\V}{\mathcal{V}}

\newcommand{\vgap}{\vspace{.1in}}

\newcommand{\tz}{\tilde z}

\newcommand{\bi}{\begin{itemize}}
\newcommand{\ei}{\end{itemize}}
\newcommand{\ba}{\begin{array}}
\newcommand{\ea}{\end{array}}

\begin{document}

\title{ A Partially Inexact Alternating Direction Method of Multipliers and its Iteration-Complexity Analysis}

\author{ Vando A. Adona \thanks{
              IME, Universidade Federal de Goi\'as, Goi\^ania, GO 74001-970, BR.
              (E-mails: {\tt vandoadona@gmail.com},  {\tt maxlng@ufg.br}   and {\tt jefferson@ufg.br}).  
              The work of these authors was supported in part by CAPES,  CNPq Grants 302666/2017-6 and 406975/2016-7.}    
  \and
      Max L.N. Gon\c calves \footnotemark[1]
    \and
Jefferson G. Melo \footnotemark[1]
}

\date{May 17,2018}

\maketitle

\begin{abstract}
This paper proposes  a partially inexact  alternating direction method of multipliers  for computing approximate solution of a linearly constrained convex optimization problem. 
This  method allows its first subproblem  to be solved inexactly using a relative approximate criterion, whereas a proximal term is added to its second subproblem in order to simplify it.
A stepsize parameter is included  in the updating rule of the Lagrangian multiplier  to 
 improve its computational  performance. 
 Pointwise and ergodic interation-complexity bounds for the proposed method are established.
 %The complexity analysis  is based on  showing that  the proposed method falls within   the setting of a hybrid proximal extragradient framework whose %iteration-complexity bounds have been recently established  in the literature.  
To the best of our knowledge, this is the first time that complexity results for  an inexact ADMM with relative error criteria has been analyzed.
 Some preliminary numerical experiments are  reported  to illustrate the advantages of the new method.
\\
  \\
  2000 Mathematics Subject Classification: 
47H05, 49M27, 90C25, 90C60,  65K10.
\\
%   \bigskip 
\\   
Key words: alternating direction method of multipliers,  relative error criterion,  hybrid extragradient method,  convex program, pointwise iteration-complexity, ergodic iteration-complexity.
\end{abstract}
%%%%%%%%%%%%%%%%%%%%%%%%%%%%%%%%%%%%%%%%%%%%%%%%%%%%%%%%
\section{Introduction} \label{sec:int}
 In this paper, we  propose and analyze a partially inexact alternating direction method of multipliers (ADMM) for  computing approximate solutions of  a linearly constrained convex optimization problem.
%Many applications can be posed in the problem formulation considered here, namely, .... 
Recently, there  has been  some growing interest in the ADMM \cite{0352.65034,0368.65053}  due to its  efficiency for solving the aforementioned class of problems; see, for instance,  \cite{Boyd:2011} for a complete review.

%Many variants of the ADMM have been studied  in the literature. %; see \cite{Proxterm, theta, acceleration, two update of the multipliers} to name just a few.   
%%These variants benefit of some special properties that a particular application may have. 
%Some of these variants \cite{HeLinear,He2015,Deng1,MJR2,He2002} included  proximal terms in the subproblems of the ADMM in order to make them  easier to solve or even have  closed-form solutions.
%Others \cite{MJR2,Cui,Gu2015} added a stepsize parameter in the Lagrangian multiplier updating  to improve the performance of the method. 
% Others  focused on studying inexact versions of the ADMM with different error conditions; for instance, 
%\cite{Eckstein2017App,Eckstein2017Relat,Xie2017} analyzed variants   whose   subproblems are solved inexactly  
%using  relative error criteria. 
% Summable error conditions were also considered in \cite{Eckstein2017App,MR1168183}; however,  it was observed in  \cite{Eckstein2017App}  that, in general, relative error conditions are more interesting from a computational viewpoint. The aforementioned relative error criteria were derived from the one considered in \cite{EcksteinSilva2013} to study inexact augmented Lagrangian method. The latter work, on the other hand, was motivated by \cite{Sol-Sv:hy.ext,Sol-Sv:hy.proj}, where the authors proposed inexact proximal-point type methods based on relative error criteria. 

 Many variants of the ADMM have been studied  in the literature. %; see \cite{Proxterm, theta, acceleration, two update of the multipliers} to name just a few.   
%These variants benefit of some special properties that a particular application may have. 
Some of these variants included  proximal terms in the subproblems of the ADMM in order to make them  easier to solve or even to have  closed-form solutions.
Others  added a stepsize parameter in the Lagrangian multiplier updating  to improve the performance of the method; see, for example,  \cite{attouch:hal,Xu2007,HeLinear,He2015,Deng1,MJR2,He2002,Cui,Gu2015}  for papers in which  one or both of the above two strategies are used. 
 Other works  focused on studying inexact versions of the ADMM with different error conditions; for instance, 
\cite{Eckstein2017App,Eckstein2017Relat,Xie2017} analyzed variants   whose   subproblems are solved inexactly  
using  relative error criteria. 
 Summable error conditions were also considered in \cite{Eckstein2017App,MR1168183}; however,  it was observed in  \cite{Eckstein2017App}  that, in general, relative error conditions are more interesting from a computational viewpoint. The aforementioned relative error criteria were derived from the one considered in \cite{EcksteinSilva2013} to study inexact augmented Lagrangian method. The latter work, on the other hand, was motivated by \cite{Sol-Sv:hy.ext,Sol-Sv:hy.proj}, where the authors proposed inexact proximal-point type methods based on relative error criteria.

%\newpage 
 %The relative error criteria considered in the aforementioned papers were motivated by the one  proposed in \cite{svaiter e solodov}, where the authors studied inexact proximal-point type methods. 
 %The  latter papers were motivated by \cite{svaiter e solodov}, where relative error criteria were considered in the setting of   proximal point-type methods.

  The contributions of this paper are threefold: 
  \begin{enumerate}
  \item[(1)] to propose  an ADMM variant which combines three of the aforementioned  strategies. Namely, (i) the first subproblem of the method  is allowed to be  solved inexactly in such a way that  a relative approximate criterion is satisfied; (ii)  a general proximal term is  added into the second subproblem; (iii) a stepsize parameter is included  in the updating rule of the Lagrangian multiplier;
  \item[2)]  to  provide   pointwise and ergodic iteration-complexity bounds for the proposed method;
   \item[3)]  to illustrate, by means of numerical experiments, the efficiency of the new method for solving some real-life applications.

   %and compare its performance  with the inexact ADMM with relative error criterion in \cite{Ecs and chines}.
  \end{enumerate}

%  \newpage 
%  The inexact partially ADMM  studied here combines three of the aforementioned  strategies. Namely, (i) the first subproblem of the method  is allowed to be  solved inexactly in such a way that  a relative approximate criterion is satisfied; (ii)  a general proximal term is  added into the second subproblem; (iii) a stepsize parameter is included  in the updating rule of the Lagrangian multiplier. The main theoretical contributions are to  provide   pointwise and ergodic iteration-complexity analyses of the proposed method.  The complexity analyses  are based on  showing that  the proposed method falls within   the setting of a hybrid proximal extragradient framework whose iteration-complexity bounds were established in \cite{ExtendindRenato}. 
% Some preliminary numerical experiments are reported to illustrating the applicability of the new method, and comparisons with the inexact ADMM with relative error criterion in \cite{Ecs and chines} are discussed.
   
 Iteration-complexity results have been considered in the literature for most of exact ADMM variants. Paper \cite{monteiro2010iteration} presented  an ergodic iteration-complexity analysis of the  ADMM. Subsequently,  \cite{HeLinear} and \cite{He2015} analyzed  ergodic and pointwise iteration-complexities of a partially proximal ADMM, respectively. We refer the reader to \cite{adona2017iteration,goncalves2017pointwise,MJR,Goncalves2018,Bot2017,Hager,GADMM2015,teboulle2014} where iteration-complexities of other ADMM variants have been considered. The complexity analyses of the present paper are based on  showing that  the proposed method falls within   the setting of a hybrid proximal extragradient framework whose iteration-complexity bounds were established in \cite{MJR2}.
To the best of our knowledge, this work is the first one to present  iteration-complexity results for an inexact ADMM with relative error. 

This paper is organized as follows.  Section~\ref{sec:basandHPE} contains some preliminary results and it is divided into two subsections. The first subsection presents our notation and basic definitions while the second one recalls a modified HPE framework and its basic iteration-complexity results. Section~\ref{subsec:Admm1} introduces the partially inexact proximal ADMM and establishes its iteration-complexity bounds. Section~\ref{sec:exp} is devoted to the numerical experiments.

\section{Preliminary Results} \label{sec:basandHPE}
This section is divided into two subsections. The first one  presents our notation and basic results. The second subsection  recalls a modified  HPE framework and  its iteration-complexity bounds.

\subsection{Notation and Basic Definitions} \label{sec:bas}

This section presents some definitions, notation and basic results used in this paper.

The $p-$norm ($p\geq1$) and maximum norm of $z\in\R^n$ are denoted, respectively, by $\|z\|_p=\left(\sum_{i=1}^{n}|z_i|^p\right)^{1/p}$ 
and $\|z\|_{\infty}=\max\{|z_1|,\dots,|z_n|\}$, when $p=2$, we omit the indice $p$. 
Let  $\V$ be a finite-dimensional  %The seminorm induced by a self-adjoint positive semidenite linear operator H : is denoted by 
real vector space with inner product and associated norm denoted by $\inner{\cdot}{\cdot}$ and $\|\cdot\|$, respectively.
For a given  self-adjoint positive semidefinite linear operator $Q:\V\to \V$, 
the seminorm induced by $Q$ on $\V$ is defined by $\|\cdot\|_{Q}= \langle Q (\cdot), \cdot\rangle ^{1/2}$.
Since $\langle Q (\cdot), \cdot\rangle$ is symmetric and bilinear, for all $v,\tilde{v}\in\V$, we have 
\begin{equation}\label{fact}
2\left\langle Qv,\tilde{v}\right\rangle\leq \norm{v}_{Q}^{2}+\norm{\tilde{v}}_{Q}^{2},\qquad  \|v+v'\|_{Q}^2\leq 2\left(\|v\|_{Q}^2+\|v'\|_{Q}^2\right).
\end{equation}
Given a set-valued operator $T:\V\tos \V$, its domain and graph  are defined, respectively, as
\[  
\mbox{Dom}\,T=\{v\in \V\,:\, T(v)\neq \emptyset\} \qquad \mbox{and}\qquad  
Gr(T)=\{ (v,\tilde{v})\in \V\times \V\;|\; \tilde{v} \in T(v)\}.
\]
%, and  its {inverse} operator $T^{-1}:\V\tos \V$ is given by
%$
%T^{-1}(\tilde{v}):=\{v\in\V \;:\; \tilde{v}\in T(v)\}.
%$
The operator $T$ is said to be   monotone iff 
\[
\inner{u-v}{\tilde{u}-\tilde{v}}\geq 0\qquad \forall \;  (u,\tilde{u}),\, (v,\tilde{v}) \,\in \, Gr(T).
\]
Moreover, $T$ is maximal monotone iff it is monotone and there is no other monotone operator $S$ such that $Gr(T)\subset Gr(S)$.
Given a scalar $\varepsilon\geq0$, the 
 {$\varepsilon$-enlargement} $T^{[\varepsilon]}:\V\tos \V$
 of a monotone operator $T:\V\tos \V$ is defined as
\begin{align}
\label{eq:def.eps}
 T^{[\varepsilon]}(v)
 =\{\tilde{v}\in \V\,:\,\inner{\tilde{v}-\tilde{u}}{v-u}\geq -\varepsilon,\;\;\forall (u,\tilde{u})\in Gr(T)\} \quad \forall\, v \in \V.
\end{align}
The 
{$\varepsilon$-subdifferential} of a 
 proper closed convex function $f:\V\to  [-\infty,\infty]$
is defined by
\[
\partial_{\varepsilon}f(v)=\{u\in \V\,:\,f(\tilde{v})\geq f(v)+\inner{u}{\tilde{v}-v}-\varepsilon,\;\;\forall \,\tilde{v}\in \V\} \qquad\forall\, v\in \V.
\]
When $\varepsilon=0$, then $\partial_0 f(v)$ 
is denoted by $\partial f(v)$
and is called the {subdifferential} of $f$ at $v$.
It is well-known that the subdifferential  operator  of a proper closed convex function is maximal monotone~\cite{Rockafellar}.
%The domain of $f$ is denoted by $\dom f$ and 
%the conjugate of $f$ is the function
%$f^*:\V \to [-\infty,\infty]$ defined as
%\[
%f^*(v) = \sup_{z \in \V} \left(\inner{v}{z} - f(z) \right)\quad \forall v \in \V.
%\]

The next result is a consequence of the transportation formula in \cite[Theorem 2.3]{Bu-Sag-Sv:teps1} combined with 
\cite[Proposition 2(i)]{Bu-Iu-Sv:teps}.

\begin{theorem}\label{for:trans}
Suppose $T:\V\tos \V$ is maximal monotone and let $\tilde{v}_i, v_i \in\V$, for $i = 1,\cdots,k$, be
such that $v_i \in T(\tilde{v}_i)$ and define
\[
\tilde{v}_{k}^{a}=\frac{1}{k}\sum_{i=1}^{k}\tilde{v}_i, \qquad v_{k}^{a}=\frac{1}{k}\sum_{i=1}^{k}v_i,\qquad 
\varepsilon_{k}^{a}=\frac{1}{k}\sum_{i=1}^{k}\inner{v_i}{\tilde{v}_i-\tilde{v}_{k}^{a}}.
\]
Then, the following hold:
\begin{itemize}
\item[(a)] $\varepsilon_{k}^{a}\geq 0$ and $v_{k}^{a}\in T^{[\varepsilon_{k}^{a}]}(\tilde{v}_{k}^{a})$;
\item[(b)] if, in addition, $T = \partial f$ for a proper closed and convex function $f$, then 
$v_{k}^{a}\in \partial_{\varepsilon_{k}^{a}}f(\tilde{v}_{k}^{a})$.
\end{itemize}
\end{theorem}
%==============================================
%                 Nova Seção 
%==============================================
\subsection
{A Modified HPE Framework} \label{sec:smhpe}

%This subsection describes and derives pointwise and ergodic convergence rate bounds for a modified HPE framework for solving monotone inclusion problems.
%Subsection~\ref{sec:NE-HPE} 
%This section describes a modified HPE framework and its corresponding pointwise and ergodic  iteration-complexity bounds for computing approximate solution of a monotone inclusion problem. This framework will be used to analyze the iteration-complexity of the  ADMM  proposed in Section~\ref{subsec:Admm1}.

Our problem of interest in this section is the monotone inclusion problem 
\begin{align}\label{eq:inc.p}
 0\in T(z),
\end{align}
where $T:\Z\tos \Z$ is a maximal monotone operator and 
 $\Z$ is  a finite-dimensional real vector space.
 We  assume that  the solution set  of~\eqref{eq:inc.p}, denoted by $T^{-1}(0)$, is nonempty.

The modified HPE framework  for computing approximate solutions of \eqref{eq:inc.p} is formally described as follows.
This framework was first considered   in  \cite{MJR2} in a more general setting. 

\vgap
\vgap
\noindent
{\bf Modified HPE framework}
\\
Step 0. Let $z_0 \in \Z$, $\eta_0 \in \R_{+}$, $\sigma \in [0, 1[$ and a self-adjoint 
 positive semidefinite linear operator   $M: \Z \to \Z $ be given, and set $k=1$.
 \\
Step 1. Obtain $(z_k,\tilde{z}_k,\eta_k) \in \Z \times \Z \times \mathbb{R}_{+}$   such that 
\begin{align}\label{breg-cond1}
 M (z_{k-1}-z_{k})  \in T(\tz_k), \qquad    
 \|{\tz}_k-z_{k}\|^2_{M} +\eta_k &\leq \sigma \|{\tz}_k-z_{k-1}\|_{M}^{2}+\eta_{k-1}.
\end{align}
Step 2. Set $k\leftarrow k+1$ and go to step 1.
\noindent
\\
\begin{remark} (i) The modified HPE framework is a generalization of the proximal point method. Indeed, 
if $M=I$ and $\sigma=\eta_0 = 0$, then \eqref{breg-cond1}   implies that $\eta_k = 0$,
$z_k =\tilde{z}_k$ and $0\in z_k-z_{k-1} +  T({z}_k)$  for every $k\geq 1$, which corresponds to  the proximal point method to solve   problem \eqref{eq:inc.p}.
%(ii) The use  of positive semidefinite matrix $M$ instead of positive definite is essential in order to analyze some ADMMs in the modified HPE %framework setting, see for example, \cite{extending???}
%It is worth mentioning that the fact that $M$ is  positive semidefinite is crucial  to include  some ADMMs in the modified HPE framework setting, see %for example, \cite{extending???}. 
%\eqref{breg-subpro} reduces to an iteration of the proximal point method (in the metric $\|\cdot\|_{M}$) applied to \eqref{eq:inc.p}.
%Third, if $w$ is strongly convex on $Z$ and $\sigma= 0$, then (\ref{breg-cond1}) implies that $\varepsilon_k= 0$
%and $z_k = \tilde z_k$ for every~$k$, and hence that $r_k \in T(z_k)$ in view of \eqref{breg-subpro}.
(ii)  In Section~\ref{subsec:Admm1}, we propose a partially inexact proximal ADMM and show that it  falls within  the modified HPE framework setting. In particular, it is specified how the triple  $(z_k,\tilde{z}_k, \eta_k)$ can be computed in this context.  It is worth mentioning that the 
 use  of a positive semidefinite operator $M$ instead of a positive definite is essential in the analysis of Section~\ref{subsec:Admm1} (see \eqref{def:matrixM1}).
More examples of algorithms which can be seen as special cases of  HPE-type frameworks can be found in \cite{Sol-Sv:hy.ext,monteiro2010iteration,monteiro2010complexity}. 
\end{remark}

%In order to state   pointwise and ergodic  iteration-complexity  results for the modified HPE %framework, the following quantity is needed:
%\begin{equation}\label{d0HPE}
%d_{0} = \inf \{ \|z^*-z_{0}\|_{M}^{2} : z^* \in T^{-1}(0)\}.
%\end{equation}
%Note that $d_0$ can be seen as a parameter that measures the quality of the initial point %$z_0$, since  $d_0=0$ if and only if $Mz_0 \in M T^{-1}(0)$. In particular,  $d_0=0$ if  $z_0$ is a solution of \eqref{eq:inc.p}. 
%It is worth mentioning that  although the most interesting situation seems to be  when the operator $M$ is positive definite,  by allowing  the latter operator ??Matrix?? be positive semidefinite, we are able to include  some ADMMs in the modified HPE framework setting, see for example, \cite{extending???}. 

We first present a  pointwise  iteration-complexity  bound for the modified HPE framework,
whose  proof can be found in  \cite[Theorem~2.2]{adona2017iteration} (see also \cite[Theorem~3.3]{MJR2}  for a more general result).

\begin{theorem}\label{th:pointwiseHPE}
Let $\{(z_k,\tilde z_k,\eta_k)\}$ be  generated by the modified HPE framework.
 Then, for every $k \ge 1$, we have $ M (z_{k-1}-z_{k})\in T(\tz_k)$ and  there exists $i\leq k$ such that
  \[
  \|z_{i-1}-z_{i}\|_M \leq
\frac{1}{\sqrt{k}} \sqrt{\frac{2(1+\sigma)d_0+4\eta_0}{1-\sigma}},
\]
where $d_{0} = \inf \{ \|z^*-z_{0}\|_{M}^{2} : z^* \in T^{-1}(0)\}$.
\end{theorem}
\begin{remark}\label{Rem:er} For a given tolerance $\bar \rho>0$, it follows from Theorem~\ref{th:pointwiseHPE} that in at most $\mathcal{O}(1/\bar\rho^2)$ iterations, the modified HPE framework computes an approximate solution $\tilde z$ of  \eqref{eq:inc.p} 
and a residual $r$ in the sense that $M r\in T(\tilde z)$ and $\|r\|_M\leq \bar \rho$.
%with a controled residual $Mr$, i.e.,   $M r\in T(\tilde z)$ and $\|r\|_M\leq \bar \rho$. 
Although $M$ is assumed to be only semidefinite positive, if $\|r\|_M=0$, then $M^{1/2} r =0$ which, in turn, implies  that $Mr=0$.
Hence, the latter inclusion implies that $\tilde z$ is a solution of problem~\eqref{eq:inc.p}. Therefore, the aforementioned concept of approximate solutions makes sense.
\end{remark}

We now state an ergodic  iteration-complexity  bound for the modified HPE framework,
whose  proof can be found in  \cite[Theorem~2.3]{adona2017iteration} (see also \cite[Theorem~3.4]{MJR2}  for a more general result).

\begin{theorem}\label{th:ergHPE}
Let $\{(z_k,\tilde z_k,\eta_k)\}$ be  generated by the modified HPE framework.
Consider the ergodic sequence $\{(\tilde z_k^a, r^a_{k},\varepsilon^a_{k})\}$ defined by
\begin{equation*}
 \tilde z^a_{k} = \frac{1}{k}\sum_{i=1}^k\tilde z_i,\quad r^a_{k} = \frac{1}{k}\sum_{i=1}^k(z_{i-1}-z_{i}), \quad
\varepsilon^a_{k} = \frac{1}{k} \sum_{i=1}^k\inner{M(z_{i-1}-z_{i})}{\tilde{z}_i -\tilde z^a_{k}},\quad \forall k\geq 1.
\end{equation*} 
Then, for every $k\geq 1$, there hold $\varepsilon^a_{k}\geq 0$, $Mr^a_k\in  T^{[\varepsilon_k^a]}(\tz^a_k)$ and
\begin{align*}
 \|r^a_k\|_M &\le \frac{2\sqrt{d_0+\eta_0}}{k}, \qquad
 \varepsilon^a_{k} \leq 
\frac{3(3-2\sigma) (d_{0} +\eta_0)}{2(1-\sigma)k},
\end{align*}
where $d_0$ is as defined in Theorem~\ref{th:pointwiseHPE}.
\end{theorem}
\begin{remark} 
For a given tolerance $\bar \rho>0$, Theorem~\ref{th:ergHPE} ensures that 
 in at most $\mathcal{O}(1/\bar \rho)$ iterations of the modified HPE framework, the triple $(\tilde z,r,\varepsilon):=(\tilde z^a_{k},r^a_{k},\varepsilon^a_{k})$ satisfies $Mr \in T^\varepsilon(\tilde z)$ and $\max\{\|r\|_M,\varepsilon\}\leq~\bar\rho$. Similarly to  Remark~\ref{Rem:er},  we see that $\tilde z$ can be interpreted as   an approximate solution of \eqref{eq:inc.p}. Note  that, the above ergodic complexity bound  is better than  the pointwise one   by  a factor of $\mathcal{O}(1/\bar \rho)$; however, the  above inclusion  is, in general,  weaker than that of the pointwise case.
\end{remark}
%%%%%%%%%%%%%%%%%
%%%%%%%%%%%%%%%%%%%%%%%%
%%%%%%%%%%%%%%%%%%%%%%%%
\section{A Partially  Inexact Proximal ADMM and its  Iteration-Complexity Analysis}\label{subsec:Admm1}

Consider the following linearly constrained problem
\begin{equation} \label{optl}
%\inf \{ f(y) + g(s) : C y + D s = c \}
\min \{ f(x) + g(y) : A x + B y =b \},
\end{equation}
where $\X$,  $\Y $ and $\Gamma$ are finite-dimensional real inner product vector spaces, $f:\X\to \bar\R$ and $g:\Y\to \bar\R$ are proper, closed and convex functions, $A:\X\to \Gamma$ and $B:\Y\to \Gamma$ are linear operators, and $b \in \Gamma$.   
%We  assume that  the solution set  of~\eqref{optl} is nonempty.

In this section, we propose a partially inexact proximal ADMM for computing approximate solutions of~\eqref{optl} and establish  pointwise and ergodic  iteration-complexity bounds  for it.  

We begin by formally stating the method.
%The partially  inexact proximal ADMM is formally described as follows.
\\ [2mm]
{\bf Partially Inexact Proximal ADMM}
 \\[2mm]
Step 0. Let an initial point $(x_0,y_0,\gamma_0) \in \X\times \Y\times \Gamma$, a penalty parameter $\beta>0$,  error tolerance parameters $\tau_1, \tau_2 \in [0,1[$, and a
self-adjoint positive semidefinite linear operator $H: \Y \to \Y$  be given.
Choose a  setpsize parameter
\begin{equation}\label{cond:theta}
\theta\in \left]0,\frac{1-2\tau_1+\sqrt{(1-2\tau_1)^2+4(1-\tau_1)}}{2(1-\tau_1)}\right[,
\end{equation}
 and set $k=1$.
  \\
Step 1. Compute  $(v_k, \, \tilde{x}_k )\in \X\times \X$  such that 
%\begin{equation}\label{cond:inex}
%v_k \in \partial f(\tilde x_k) - A^*\tilde{\gamma}_{k}, \qquad(1-\tau_2)\|\tilde x_k-x_{k-1}\|^2+ 2\beta |\inner{\tilde x_k-x_{k-1}}{v_k}|+\beta^2\|v_k\|^2\leq\tau_1\| \tilde{\gamma}_{k}-{\gamma}_{k-1}\|^2,
%\end{equation}
\begin{equation}\label{cond:inex}
v_k \in \partial f(\tilde x_k) - A^*\tilde{\gamma}_{k}, \qquad \|\tilde x_k-x_{k-1}+\beta v_k\|^2\leq\tau_1\| \tilde{\gamma}_{k}-{\gamma}_{k-1}\|^2+\tau_2\|\tilde x_k-x_{k-1}\|^2,
\end{equation}
where
\begin{equation} \label{xtilde1}
 \tilde{\gamma}_{k}={\gamma}_{k-1}-\beta(A\tilde x_k +B y_{k-1} - b),
\end{equation}
and compute an optimal solution $y_k\in \Y$ of the subproblem
\begin{equation} \label{def:tyk-admm1}
\min_{y \in \Y} \left \{ g(y) - \inner{ {\gamma}_{k-1}}{By} +
\frac{\beta}{2} \|   A \tilde{x}_k+B y - b \|^2 +\frac{1}{2}\|y- y_{k-1}\|_{H}^2\right\}.
\end{equation}
Step 2. Set 
\begin{equation}\label{admm:eqxk1}
x_k = x_{k-1}-\beta v_k, \qquad \gamma_k = \gamma_{k-1}-\theta\beta\left(A\tilde{x}_k+By_k-b\right)
\end{equation}
and $k \leftarrow k+1$, and go to step~1.
\\
%%%%%%%%
\begin{remark}\label{remark23}
(i) If $\tau_1=\tau_2=0$, then $\tilde x_k=x_k$ due to the inequality in \eqref{cond:inex} and the first relation in \eqref{admm:eqxk1}. Hence,  since  $v_k=(x_{k-1}-x_k)/\beta$,  the first subproblem of Step~1 is equivalent to compute an exact solution $x_k\in \X$ of the following subproblem
%we see that  \eqref{cond:inex} is equivalent to
\begin{equation} \label{eq:probx_exact}
\min_{x \in \X} \left \{ f(x) - \inner{ {\gamma}_{k-1}}{Ax} +
\frac{\beta}{2} \|   A x+B y_{k-1} - b \|^2 +\frac{1}{2\beta}\|x-x_{k-1}\|^2\right\},
\end{equation}
and then  the partially inexact proximal ADMM  becomes the  proximal ADMM with stepsize $\theta \in \, ]0,(1+\sqrt{5})/2[$ and proximal terms given by $(1/\beta)I$ and $H$. Therefore, the proposed method can be seen as an extension of the proximal ADMM, which  subproblem \eqref{eq:probx_exact} is   solved inexactly  using  a relative approximate criterion. (ii) Subproblem \eqref{def:tyk-admm1} contains a proximal term defined by a self-adjoint positive semidefinite linear operator
 $H$ which, appropriately chosen, makes the subproblem easier to solve or even to have closed-form solution.
%  or even has a closed-form solution.
For instance, if  $H=sI -\beta B^* B$ with $s>\beta\|B\|^2$,   subproblem  \eqref{def:tyk-admm1} is equivalent to
\[
\min_{y \in \Y} \left \{ g(y) +
\frac{s}{2} \| y-\bar y   \|^2 \right\},
\]
for some $\bar y \in \Y$, which has a closed-form solution when $g(\cdot)=\|\cdot\|_1$.   (iii) The use of a relative approximate criterion in  \eqref{def:tyk-admm1}
 requires, as far as we know,  the  stepsize parameter $\theta\in ]0,1]$. However, since,  in many applications, the second subproblem \eqref{def:tyk-admm1}  is   solved  exactly  %(one can also use the operator $H$ to obtain  this exact solution) 
 and a stepsize parameter $\theta>1$ accelerates the method, here  only the first subproblem is assumed to be solved inexactly.
 (iv) The  partially inexact proximal ADMM is close related to  \cite[Algorithm~2]{Eckstein2017App}. Indeed, the latter method corresponds to  the former one with $H=0$,  $\theta=1$ and the following condition 
\begin{equation}\label{er:EC}
2\beta |\inner{\tilde x_k-x_{k-1}}{v_k}|+\beta^2\|v_k\|^2\leq \tau_1\| \tilde{\gamma}_{k}-{\gamma}_{k-1}\|^2
\end{equation}
instead of the inequality  in \eqref{cond:inex}. Numerical comparisons between the  partially inexact proximal ADMM and Algorithm~2 in \cite{Eckstein2017App} will be provided   in  Section~ \ref{sec:exp}.

%The algorithm~2  studied in  \cite{???} corresponds to  the
%partially inexact proximal ADMM with $H=0$,  $\theta=1$ and the following condition 
%$
%2\beta |\inner{\tilde x_k-x_{k-1}}{v_k}|+\beta^2\|v_k\|^2\leq \tau_1\| \tilde{\gamma}_{k}-{\gamma}_{k-1}\|^2
%$
%instead of the inequality  in \eqref{cond:inex}. As will be seen in~Section~\ref{}, algorithm~2  in  \cite{???} and the partially inexact proximal ADMM with $\theta=1$
%have similar numerical performances. However, our proposed method with $\theta \geq 1$ outperforms the algorithm~2
\end{remark}
%\textcolor{red}{As already mentioned,  a balance between the parameters $\theta$ and $\tau_1$ must be taken into account, since 
% the choice of $\theta$ close to $(1+\sqrt{5})/2$, which is the best one of the exact ADMM for many applications, implies a more restrictive  relative error criterion  in \eqref{cond:inex}.
%Since the choice of a larger $\theta$  implies a more restrictive  relative error criterion  in \eqref{cond:inex},
%a balance between  $\theta$ and the error tolerance parameter $\tau_1$ may be the best strategy.
%}

In the following, we proceed to provide iteration-complexity bounds for  the partially inexact proximal ADMM. Our analysis is done by  showing that  it is an instance of the modified HPE framework for computing approximate solutions of the monotone inclusion problem
\begin{equation} \label{FAB}
0\in T(x,y,\gamma) = \left[ \begin{array}{c}  \partial f(x)- A^{*}\gamma \\ \partial g(y)- B^{*}\gamma \\ Ax+By-b
\end{array} \right].
\end{equation}
We  assume that  the solution set  of~\eqref{FAB}, denoted by $\Omega^*$, is nonempty.
The iteration-complexity results will follow immediately  from Theorems \ref{th:pointwiseHPE} and \ref{th:ergHPE}. Let us now introduce the elements required by the setting of Section \ref{sec:smhpe}. Namely,
consider  the vector space $\Z=\X\times\Y\times \Gamma$ and the  self-adjoint positive semidefinite linear operator   
\begin{equation}\label{def:matrixM1}
M=\left[ 
\begin{array}{ccc} 
I/\beta &0&0\\
0&(H+{\beta} B^*B)&0\\[2mm]
0&0&{I}/(\theta\beta)
\end{array} \right].
\end{equation}
In this setting,  the quantity $d_0$ defined in Theorem~\ref{th:pointwiseHPE}  becomes  
\begin{equation}\label{def:d0admm1}
d_0=\inf\left\{\|(x-x_0,y-y_0,\gamma-\gamma_0)\|^{2}_{M}: {(x,y,\gamma) \in T^{-1}(0)}\right\}.
\end{equation}

We start by presenting a preliminary technical result, which basically shows that a certain sequence generated by  the partially inexact  proximal ADMM satisfies the inclusion in \eqref{breg-cond1} with $T$ and $M$ as above.

%%%%%%%%%
\begin{lemma} \label{lem:inclusion}
Consider $(x_k,y_k,\gamma_k)$  and $(\tilde x_k,\tilde{\gamma}_k)$ generated at  the k-iteration of the partially inexact  proximal ADMM.
Then,
\begin{eqnarray}
\ \frac{1}{\beta}(x_{k-1}-x_k)&\in& \partial f(\tilde{x}_k)-A^*\tilde{\gamma}_k,  \label{aux.01}\\
(H+\beta B^*B)(y_{k-1}-y_k)&\in& \partial g(y_k)-B^*\tilde\gamma_k,\label{aux.21}\\
 \frac{1}{\theta\beta}(\gamma_{k-1}-\gamma_k)&=& A\tilde{x}_k+By_k-b.\label{aux.11}
\end{eqnarray}
As a consequence,  $z_k=(x_k,y_k, \gamma_k)$ and $\tilde z_k=(\tilde{x}_k,y_k,\tilde \gamma_k)$  
satisfy  inclusion~\eqref{breg-cond1} with $T$ and $M$ as in \eqref{FAB} and \eqref{def:matrixM1}, respectively.
\end{lemma}
\begin{proof} Inclusion \eqref{aux.01} follows trivially from the inclusion in \eqref{cond:inex} and the first relation  in \eqref{admm:eqxk1}.
Now, from the optimality condition of \eqref{def:tyk-admm1} and  the definition of $\tilde\gamma_k$ in \eqref{xtilde1}, we obtain
\begin{align*}
0 &\in \partial g(y_k)-B^*\gamma_{k-1}+\beta B^*(A\tilde{x}_k+By_k-b)+H(y_k- y_{k-1})\\[2mm]
   &= \partial g(y_k)-B^*[\gamma_{k-1}-\beta (A\tilde{x}_k+B{y}_{k-1}-b)]+\beta B^*B(y_{k}-{y}_{k-1})+H(y_k- y_{k-1}) \\[2mm]
   & = \partial g(y_k)-B^*\tilde{\gamma}_k+\beta B^*B(y_{k}-{y}_{k-1})+H(y_k-y_{k-1}).
\end{align*}
which proves to  \eqref{aux.21}.
%which, combined with  the previous inclusion, yield \eqref{aux.2}.
The relation  \eqref{aux.11} follows immediately from the second relation in \eqref{admm:eqxk1}.
To end the proof, note that the last statement  of the lemma follows directly by \eqref{aux.01}--\eqref{aux.11} and definitions of $T$ and $M$  in \eqref{FAB} and \eqref{def:matrixM1}, respectively.
\qed
\end{proof}
%%%%%

The following  result presents some relations satisfied by the sequences generated by the partially inexact proximal ADMM. These relations are essential  to show that the latter method is   an instance of the modified HPE framework.

\begin{lemma}\label{lem:deltak1}
Let $\{(x_k,y_k,\gamma_k)\}$  and $\{(\tilde x_k,\tilde{\gamma}_k)\}$ be generated by the partially inexact  proximal ADMM.
Then, the following hold:
 \begin{itemize}
\item[(a)]  for any $k\geq 1$, we have
\[
 \tilde{\gamma}_k-\gamma_{k-1}=\frac{1}{\theta}(\gamma_k-\gamma_{k-1})+\beta B(y_{k}-y_{k-1}), \quad  \tilde{\gamma}_k-\gamma_k=\frac{1-\theta}{\theta}(\gamma_k-\gamma_{k-1})+\beta B(y_{k}-y_{k-1});
\]
\item[(b)] we have
\[
 \frac{1}{2}\|y_1-y_0\|_{H}^2-\frac{1}{\sqrt{\theta}}\langle B(y_{1}-y_{0}),\gamma_1-\gamma_{0} \rangle   \leq  2\max\left\{{1},\frac{\theta}{2-\theta}\right\} d_0,
\] 
where  $d_0$ is as
in \eqref{def:d0admm1};
\item[(c)] for every $k\geq 2$, we have
 %\begin{equation}\label{eq:deltak}
 {
\[ \frac{1}{\theta} \left\langle{\gamma_k-\gamma_{k-1},}{B(y_k-y_{k-1})}\right\rangle
 \geq \frac{1-\theta}{\theta} \left\langle{\gamma_{k-1}-\gamma_{k-2},}{B(y_k-y_{k-1})}\right\rangle+ \frac{1}{2}\|y_k-y_{k-1}\|_{H}^2-\frac12\|y_{k-1}-y_{k-2}\|_{H}^2.
\]}
 \end{itemize} 
%\end{equation}
\end{lemma}
\begin{proof}
(a) The first relation follows by noting that the   definitions of $\tilde\gamma_k$ and   ${\gamma}_k$ in \eqref{xtilde1} and \eqref{admm:eqxk1}, respectively, yield
\begin{align*}
 \tilde{\gamma}_{k}-{\gamma}_{k-1}= -\beta(A\tilde x_k+By_{k-1}-b)=\frac{1}{\theta}(\gamma_k-\gamma_{k-1})+\beta B(y_k-y_{k-1}).
\end{align*}
The second relation in (a) follows trivially from the first one.

%This item follows trivially from \eqref{xtilde1} and \eqref{admm:eqxk1}.

%
(b) First, note that 
\begin{align*}
 \nonumber
 0 \leq \dfrac{1}{2\beta}\left\|\frac{1}{\sqrt{\theta}}(\gamma_1-\gamma_{0})+ \beta B(y_{1}-y_{0})\right\|^2
 = \dfrac{1}{2\theta\beta}\| \gamma_1-\gamma_{0}\|^2+\frac{1}{\sqrt{\theta}}\langle B(y_{1}-y_{0}),\gamma_1-\gamma_{0} \rangle
 +\dfrac{\beta}{2}\| B(y_{1}-y_{0})\|^2,
\end{align*}
which, for every $z^*=(x^*,y^*,\gamma^*)\in \Omega^*$,  yields
\begin{align*}
\frac{1}{2}\|y_1-y_0\|_{H_{2}}^2-\frac{1}{\sqrt{\theta}}\langle B(y_{1}-y_{0}),\gamma_1-\gamma_{0} \rangle &\leq 
\frac{1}{2}\left(\|y_1-y_0\|_{H_{2}}^2+ \frac{1}{\theta\beta}\|\gamma_1-\gamma_{0}\|^2
 +\beta\| B(y_{1}-y_{0})\|^2\right)\\
 &\leq
\|y_1-y^*\|_{H_{2}}^2+
\|y_0-y^*\|_{H_{2}}^2+
 \frac{1}{\theta\beta}\|\gamma_1-\gamma^*\|^2\\
&+ \frac{1}{\theta\beta}\| \gamma_0-\gamma^*\|^2+
\beta\| B(y_{1}-y^*)\|^2 +
\beta\| B(y_{0}-y^*)\|^2,
\end{align*}
where the last inequality is due to the second property  in  \eqref{fact}.
Hence,  using \eqref{def:matrixM1}, we obtain 
\begin{equation}
 \label{eq_0000000121}
 \frac{1}{2}\|y_1-y_0\|_{H_{2}}^2-\frac{1}{\sqrt{\theta}}\langle B(y_{1}-y_{0}),\gamma_1-\gamma_{0} \rangle\leq  
 \|z_1-z^*\|^2_{M} +\|z_0-z^*\|^2_{M},
\end{equation}
where $z_0=(x_0,y_0,\gamma_0)$ and $z_1=(x_1,y_1,\gamma_1)$. On the other hand, from Lemma~\ref{lem:inclusion}  with $k=1$, we have $M(z_0-z_1)\in T(\tilde z_1)$, where $\tilde z_1=(\tilde{x}_1,y_1,\tilde \gamma_1)$ and $T$
is as in \eqref{FAB}. Using this fact
and the monotonicity of $T$, we obtain $\langle \tilde z_1-z^*, M(z_0-z_1)\rangle\geq 0$ for all $z^*=(x^*,y^*,z^*)\in \Omega^*$. Hence,
\begin{align}
 \nonumber
 \|z^*-z_0\|_{M}^2 - \|z^*-z_1\|_{M}^2 &= \|\tilde{z}_1-z_0\|_{M}^2-\|\tilde{z}_1-z_1\|_{M}^2
   +2\langle \tilde{z}_1-z^*, M(z_0-z_1) \rangle \label{eq_0000000123}\\
   &\geq  \|\tilde{z}_1-z_0\|_{M}^2-\|\tilde{z}_1-z_1\|_{M}^2.
\end{align}
It follows from 
\eqref{def:matrixM1}, item  (a), and some direct calculations that
%\begin{align*}
\begin{align} \label{eq:deltakx}
\|\tilde z_1&-z_1\|^2_{M}=\frac{1}{\beta}\|\tilde x_1-x_1\|^2+\frac{1}{\theta\beta}\|\tilde \gamma_1-\gamma_1\|^2=\frac{1}{\beta}\|\tilde x_1-x_1\|^2+\frac{1}{\theta\beta}\left\|\frac{1-\theta}{\theta}(\gamma_1-\gamma_{0})+ \beta B(y_{1}-y_{0})\right\|^2\nonumber\\
 &= \frac{1}{\beta}\|\tilde x_1-x_1\|^2+\frac{(1-\theta)^2}{\beta\theta^3}\| \gamma_1-\gamma_{0}\|^2+\frac{2(1-\theta)}{\theta^2}\langle B(y_{1}-y_{0}),\gamma_1-\gamma_{0} \rangle
 +\frac{\beta}{\theta}\| B(y_{1}-y_{0})\|^2.
\end{align}
Moreover,  \eqref{def:matrixM1} and item  (a) also yield
\begin{align} 
 \nonumber
& \|\tilde z_1-z_0\|^2_{M}=\frac{1}{\beta}\|\tilde x_1-x_0\|^2+\|y_1-y_0\|^2_{(\beta B^* B+H)}+\frac{1}{\theta\beta}\|\tilde \gamma_1-\gamma_0\|^2\\
  &\geq \frac{1}{\beta}\|\tilde x_1-x_0\|^2+\beta\|B(y_1-y_0)\|^2+\frac{\tau_1}{\beta}\|\tilde \gamma_1-\gamma_0\|^2+\frac{1-\tau_1\theta}{\theta\beta}\left\|\frac{1}{\theta}(\gamma_1-\gamma_{0})+\beta B(y_{1}-y_{0})\right\|^2\nonumber\\
  &= \frac{1}{\beta}\|\tilde x_1-x_0\|^2+\frac{\tau_1}{\beta}\|\tilde \gamma_1-\gamma_0\|^2+ \frac{[1+(1-\tau_1)\theta]\beta}{\theta}\|B(y_1-y_0)\|^2+\frac{1-\tau_1\theta}{\beta\theta^3}\left\|\gamma_1-\gamma_{0}\right\|^2\nonumber\\
  &+\frac{2(1-\tau_1\theta)}{\theta^2}\langle B(y_{1}-y_{0}),\gamma_1-\gamma_{0} \rangle.\label{eq:deltaky}
\end{align}
Combining the above two conclusions, we obtain
\begin{align} \label{eeq:a45}
 \|\tilde z_1-z_0\|^2_{M}-\|\tilde z_1-z_1\|^2_{M}
 &\geq \frac{1}{\beta}\left(\|\tilde x_1-x_0\|^2-\|\tilde x_1-x_1\|^2+{\tau_1}\|\tilde \gamma_1-\gamma_0\|^2\right)+(1-\tau_1)\beta\| B(y_{1}-y_{0})\|^2  \nonumber\\ 
 &+ \frac{2-\theta-\tau_1}{\beta\theta^2}\|\gamma_1- \gamma_0\|^2+
 \frac{2(1-\tau_1)}{\theta}\langle B(y_{1}-y_{0}),\gamma_1-\gamma_{0} \rangle.
  \end{align}
  Now, note that the inequality in \eqref{cond:inex} with $k=1$ and the definition of $x_1$ in  \eqref{def:tyk-admm1} imply that 
\[
0\leq\tau_2\|\tilde x_1-x_{0}\|^2 - \|\tilde x_1-x_1\|^2+ \tau_1\| \tilde{\gamma}_{1}-{\gamma}_{{0}} \|^2
\]
 which, combined with \eqref{eeq:a45} and $\tau_2 \in [0,1[$, yields  
 \begin{align} \nonumber
  \|\tilde z_1-z_0\|^2_{M}-\|\tilde z_1-z_1\|^2_{M}
 &\geq (1-\tau_1)\beta\| B(y_{1}-y_{0})\|^2 + \frac{2-\theta-\tau_1}{\beta\theta^2}\|\gamma_1- \gamma_0\|^2+
 \frac{2(1-\tau_1)}{\theta}\langle B(y_{1}-y_{0}),\gamma_1-\gamma_{0} \rangle\\
 &=\frac{1-\theta}{\beta\theta^2}\|\gamma_1- \gamma_0\|^2+(1-\tau_1)\left\|\sqrt{\beta}B(y_{1}-y_{0})+\frac{1}{\theta\sqrt{\beta}}(\gamma_1- \gamma_0)\right\|^2\nonumber\\
 &\geq\frac{1-\theta}{\beta\theta^2}\|\gamma_1- \gamma_0\|^2\nonumber.
 \end{align}
Hence,  if $\theta\in ]0,1]$, then we have
\begin{equation}\label{ineq_s23}
 \|\tilde z_1-z_1\|^2_{M}\leq \|\tilde z_1-z_0\|^2_{M}.
\end{equation}
Now, if $\theta> 1$, then  we have
\begin{align*}
 \|\tilde z_1-z_1\|^2_{M}- \|\tilde z_1-z_0\|^2_{M}&\leq\frac{\theta-1}{\beta\theta^2}\|\gamma_1- \gamma_0\|^2\\
 &\leq \frac{2(\theta-1)}{\theta}\left(\frac1{\beta\theta}\|\gamma_1- \gamma^*\|^2+\frac1{\beta\theta}\|\gamma_0-\gamma^*\|^2\right)\\
&\leq \frac{2(\theta-1)}{\theta} \left[\|z_0-{z}^*\|_{M}^2+\|z_1-{z}^*\|_{M}^2\right]
\end{align*}
where   the second inequality is due to the second property  in  \eqref{fact}, and the last inequality is due to \eqref{def:matrixM1} and definitions of $z_0,z_1$ and ${z}^*$.
Hence, combining the last estimative with \eqref{eq_0000000123},   we obtain 
$$
\|z_1-{z}^*\|_{M}^2\leq\frac{3\theta-2}{2-\theta}\|z_0-{z}^* \|_{M}^2.
$$ 
Thus, it follows from   \eqref{eq_0000000123}, \eqref{ineq_s23} and the last inequality that 
\begin{equation}\label{eq:457}
\|z_1-{z}^*\|_{M}^2\leq \max\left\{1,\frac{3\theta-2}{2-\theta}\right\}\|z_0-{z}^*\|_{M}^2.
\end{equation}
Therefore, the desired inequality follows  from  \eqref{eq_0000000121},  \eqref{eq:457}
and the definition of $d_0$ in \eqref{def:d0admm1}.

(c) From the optimality condition for \eqref{def:tyk-admm1}, the definition of $ \tilde\gamma_k$ in \eqref{xtilde1}
and item (a), we have, for every $k\geq 1$, 
\begin{align*}
 \partial g(y_k)\ni B^*( \tilde\gamma_k-\beta B(y_k-y_{k-1}))-H(y_k-y_{k-1})=\frac{1}{\theta}B^*( \gamma_k-(1-\theta)\gamma_{k-1})-H(y_k-y_{k-1}). 
\end{align*}
For any $k\geq 2$, using the above inclusion with $k \leftarrow k$ and $k \leftarrow k-1$ and the monotonicity of $\partial g$ , we obtain
%, for all $k\geq 1$,
%
\begin{align*}
&\frac{1}{\theta} \left\langle{B^*(\gamma_k-\gamma_{k-1})-(1-\theta)B^*(\gamma_{k-1}-\gamma_{k-2})},{y_k-y_{k-1}}\right\rangle\\ 
&\geq \inner{H(y_k-y_{k-1})}{y_k-y_{k-1}}-\inner{H(y_{k-1}-y_{k-2})}{y_k-y_{k-1}}\\
% &\geq \|y_k-y_{k-1}\|_{H}^2-\dfrac{1}{2}\|y_{k-1}-y_{k-2}\|_{H}^2-\dfrac{1}{2}\|y_k-y_{k-1}\|_{H}^2,\\
 &\geq \frac{1}{2}\|y_k-y_{k-1}\|_{H}^2-\dfrac{1}{2}\|y_{k-1}-y_{k-2}\|_{H}^2,
\end{align*}
where the last inequality is due to  the first property in \eqref{fact}, and so the proof of the lemma follows.
\qed
\end{proof}
%%%%%%%
%%%%%
%%%%%%%%

We next consider a technical result.

%The next result is  used to show that a certain sequence  related to the partially inexact proximal ADMM  satisfies the relative error criterion of  the modified HPE %framework.
\begin{lemma}\label{pro:sigma}
Let scalars $\tau_1, \tau_2$ and $\theta$ be as in step~0 of the partially inexact proximal ADMM. Then,  there exists a scalar $\sigma\in [\tau_2,1[$ such that the  matrix  
 \begin{equation} \label{matrixtheta>1}
G=\left[
\begin{array}{cc} 
 \sigma-1+(\sigma-\tau_1)\theta& (1-\theta)[\sigma-1+(1-\tau_1)\theta]\\[2mm]
(1-\theta)[\sigma-1+(1-\tau_1)\theta]&  \sigma-1+(2-\theta-\tau_1)\theta\\
\end{array} \right]
\end{equation}
 is positive semidefinite.  
 \end{lemma}
\begin{proof} 
%First, since the inequality of item (a) holds strictly  for  $ \sigma=1$,  we conclude that there exists $\bar\sigma \in [0,1[$ such that 
%  \[
%   [\sigma-1+(1-\tau_1)\theta]\geq 0 \quad \forall \;\sigma \in  [\bar\sigma,1).
%\] 
%On the other hand, 
Note that the matrix $G$ in \eqref{matrixtheta>1} with $ \sigma=1$ reduces to 
\[
\theta\left[ 
\begin{array}{cc} 
 1-\tau_1&(1-\theta){(1-\tau_1)}\\
(1-\theta){(1-\tau_1)}& 2-\theta-\tau_1
\end{array} \right].
\]
Using \eqref{cond:theta} and $\tau_1, \tau_2 \in [0,1[$, it can be  verified that  the above matrix is positive definite. Hence, we conclude that  there exists $\hat \sigma \in [0,1[$ such that   $G$  is positive semidefinite for all $\sigma\in [\hat \sigma,1[$.  Therefore, the lemma follows by taking 
$\sigma= \max\{\tau_2,\hat \sigma\}.$
\qed
\end{proof}

In the following, we  show that the partially inexact proximal ADMM can be regarded as an instance of the modified HPE framework.

%%%%%%%%
\begin{proposition}\label{maincorADMM1} 
Let $\{(x_k,y_k,\gamma_k)\}$  and $\{(\tilde x_k,\tilde{\gamma}_k)\}$ be generated by the partially inexact  proximal ADMM.
Let also $T$, $M$ and $d_0$   be
as in \eqref{FAB}, \eqref{def:matrixM1} and  \eqref{def:d0admm1},  respectively.
Define 
\begin{equation}\label{def:eta0}
z_0=(x_0,y_0,\gamma_0), \quad \mu=\frac{4[\sigma-1+(1-\tau_1)\theta]}{\theta^{3/2}}\max\left\{{1},\frac{\theta}{2-\theta}\right\},   \quad \eta_0=\mu d_0
\end{equation}
and, for all $k\geq 1$,
\begin{align} \label{i645}
 & z_k=(x_k,y_k,\gamma_k), \qquad  
	\tilde z_k=(\tilde{x}_k,y_k,\tilde \gamma_k),\\[2mm]
	&\eta_k= \frac{[\sigma-1+(2-\theta-\tau_1)\theta]}{\beta\theta^3}\|\gamma_{k}-\gamma_{k-1}\|^2+
\frac{[\sigma-1+(1-\tau_1)\theta]}{\theta}\|y_k-y_{k-1}\|_{H}^2,\label{etak}
\end{align}
where $\sigma \in [\tau_2,1[$ is  given by Lemma~\ref{pro:sigma}. Then,  $(z_k,\tilde{z}_k,\eta_k)$ satisfies the error condition in \eqref{breg-cond1} for every $k\geq 1$.  As a consequence,  the partially inexact  proximal ADMM is an instance of the modified HPE framework.
\end{proposition}
\begin{proof} First of all, since the matrix $G$ in \eqref{matrixtheta>1} is positive semidefinite and $\sigma\in [\tau_2,1[$, we have
\begin{equation}\label{eq:3455}
[\sigma-1+(1-\tau_1)\theta]\geq [\sigma-1+(\sigma-\tau_1)\theta]=g_{11}\geq 0.
\end{equation}
Now, using \eqref{def:matrixM1} and definitions of $\{z_{k}\}$ and $\{\tilde z_k\}$ in \eqref{i645}, we obtain
\begin{equation*}% \label{eq:1001}
 \|\tilde z_k-z_{k-1}\|_{M}^2= \frac{1}{\beta}\|\tilde{x}_k-x_{k-1}\|^2+\|y_k-y_{k-1}\|_{H}^2+\beta\|B(y_k-y_{k-1})\|^2+\frac{1}{\beta\theta}\|\tilde \gamma_k-\gamma_{k-1}\|^2
\end{equation*}
and %
\begin{align*}%\label{eq:1002}
 \|\tilde z_k-z_k\|^2_{M}&= \frac{1}{\beta}\|\tilde{x}_k-x_{k}\|^2+\frac{1}{\beta\theta}\|\tilde \gamma_k-\gamma_k\|^2.
 %=\left\|\frac{1-\theta}{\theta}(\gamma_k-\gamma_{k-1})+H_k B(y_{k}-y_{k-1})\right\|_{\Gamma,\theta^{-1}H_k^{-1}}^2.
\end{align*} 
Hence,
\begin{align}\nonumber
\sigma\|\tz_{k}-{z_{k-1}}\|_{M}^2- \|{\tz}_k-{z_k}\|_{M}^2&= \frac{1}{\beta}\left(
{\sigma}\|\tilde{x}_k-x_{k-1}\|^2-\|\tilde{x}_k-x_{k}\|^2+{\tau_1}\|\tilde \gamma_k-\gamma_{k-1}\|^2\right)+\sigma\|y_k-y_{k-1}\|_{H}^2\label{eq:325}\\
&+\sigma\beta\|B(y_k-y_{k-1})\|^2+\frac{\sigma-\tau_1\theta}{\beta\theta}\|\tilde \gamma_k-\gamma_{k-1}\|^2-\frac{1}{\beta\theta}\|\tilde \gamma_k-\gamma_k\|^2.
\end{align}
Note that the inequality in \eqref{cond:inex}  and definition of $x_k$ in  \eqref{def:tyk-admm1} imply that 
 \begin{align*}
0&\leq  \tau_2\|\tilde x_k-x_{k-1}\|^2 - \|\tilde x_k-x_k\|^2+\tau_1\| \tilde{\gamma}_{k}-{\gamma}_{{k-1}} \|^2
  \end{align*}
which, combined with \eqref{eq:325} and the fact that $\sigma\geq \tau_2$, yields 
\begin{align}\nonumber
\sigma\|\tz_{k}-{z_{k-1}}\|_{M}^2- \|{\tz}_k-{z_k}\|_{M}^2&\geq\sigma\|y_k-y_{k-1}\|_{H}^2+ \sigma\beta\|B(y_k-y_{k-1})\|^2+\frac{\sigma-\tau_1\theta}{\beta\theta}\|\tilde \gamma_k-\gamma_{k-1}\|^2\label{eq:ah34}\\
&-\frac{1}{\beta\theta}\|\tilde \gamma_k-\gamma_k\|^2.
\end{align}
On the other hand, it follows from Lemma~\ref{lem:deltak1}(a) that
\begin{align*} %\label{eq:1003}
\frac{\sigma-\tau_1\theta}{\beta\theta}&\|\tilde \gamma_k-\gamma_{k-1}\|^2-\frac{1}{\beta\theta}\|\tilde \gamma_k-\gamma_k\|^2\\&=\frac{\sigma-\tau_1\theta}{\beta\theta}\left\|\frac{1}{\theta}(\gamma_k-\gamma_{k-1})+\beta B(y_{k}-y_{k-1})\right\|^2 -\frac{1}{\beta\theta}\left\|\frac{1-\theta}{\theta}(\gamma_k-\gamma_{k-1})+\beta B(y_{k}-y_{k-1})\right\|^2\\
\nonumber
 &=\frac{\sigma-1+(2-\theta-\tau_1)\theta}{\beta\theta^3}\|\gamma_{k}-\gamma_{k-1}\|^2+\frac{(\sigma-1-\tau_1\theta)\beta}{\theta}\|B(y_{k}-y_{k-1})\|^2\\\nonumber
 &+\frac{2[\sigma-1+(1-\tau_1)\theta]}{\theta^2}\inner{\gamma_{k}-\gamma_{k-1}}{B(y_k-y_{k-1})}.
 \end{align*}
  Hence, combining the last equality and \eqref{eq:ah34}, we obtain 
\begin{align}\nonumber
\sigma\|\tz_{k}&-{z_{k-1}}\|_{M}^2- \|{\tz}_k-{z_k}\|_{M}^2\geq \sigma\|y_k-y_{k-1}\|_{H}^2+\frac{[\sigma-1+(2-\theta-\tau_1)\theta]}{\beta\theta^3}\|\gamma_{k}-\gamma_{k-1}\|^2\label{eq:e309} \\
&+\frac{[\sigma-1+(\sigma-\tau_1)\theta]\beta}{\theta}\|B(y_{k}-y_{k-1})\|^2
 +\frac{2[\sigma-1+(1-\tau_1)\theta]}{\theta^2}\inner{\gamma_{k}-\gamma_{k-1}}{B(y_k-y_{k-1})}.
\end{align}

We will now consider two cases: $k=1$ and $k>1$. 
\\[1mm]
Case 1 ($k=1$): It follows from \eqref{eq:e309} with $k=1$,  \eqref{eq:3455} and  Lemma~\ref{lem:deltak1}(b) that
\begin{align*}%\label{eq:1003}
\sigma\|\tz_{1}-{z_{0}}\|_{ M}^2-& \|{\tz}_1-{z_1}\|_{M}^2 \geq\frac{[\sigma-1+(2-\theta-\tau_1)\theta]}{\beta\theta^3}\|\gamma_{1}-\gamma_{0}\|^2+\frac{[\sigma-1+(\sigma-\tau_1)\theta]\beta}{\theta}\|B(y_{1}-y_{0})\|^2\\
& +\frac{[\sigma-1+(1-\tau_1)\theta+\sigma\theta^{3/2}]}{\theta^{3/2}}\|y_1-y_{0}\|_{H}^2-\frac{4[\sigma-1+(1-\tau_1)\theta]}{\theta^{3/2}}\max\left\{{1},\frac{\theta}{2-\theta}\right\} d_0
 \end{align*}
which, combined with definitions of $\eta_0$ and $\eta_1$, yields 
\begin{align}\nonumber
\sigma\|\tz_{1}-{z_{0}}\|_{ M}^2- \|{\tz}_1-{z_1}\|_{M}^2+\eta_0-\eta_1 &\geq\frac{[\sigma-1+(\sigma-\tau_1)\theta]}{\theta}\left(\beta\|B(y_{1}-y_{0})\|^2+\frac{1}{\sqrt{\theta}}\|y_1-y_{0}\|_{H}^2\right)\label{eq:1003}\\
& +\frac{[(1-\sigma)(1+\sqrt{\theta}-\theta)+\tau_1\theta]}{{\theta}}\|y_1-y_{0}\|_{H}^2.
 \end{align}
%\begin{align*}%\label{eq:1003}
%\sigma\|{z_{0}} -\tz_{1}\|_{ M}^2- \|{z_1}- {\tz}_1\|_{M}^2+\eta_0-\eta_1 &\geq\frac{[\sigma-1+(\sigma-\tau_1)\theta]}{\theta}\beta\|B(y_{0}-y_{1})\|^2 +\frac{[\sigma-1+(1+\sigma\sqrt{\theta}-\tau_1)\theta]}{\theta^{3/2}}\|y_{0}-y_1\|_{H}^2.
% \end{align*}
Using \eqref{cond:theta}, we have $\theta \in \, ]0,(1+\sqrt{5})/2[$ which in turn implies that $(1+\sqrt{\theta}-\theta)\geq 0$.
Hence, inequality \eqref{breg-cond1}  with $k=1$  follows from  \eqref{eq:3455}, \eqref{eq:1003}  and the fact that  $\sigma< 1$.
\\[2mm]
Case 2 ($k>1$): It follows from \eqref{eq:e309},  \eqref{eq:3455} and Lemma~\ref{lem:deltak1}(c) that
\begin{align*} %\label{eq:1003}
\sigma\|\tz_{k}-{z_{k-1}}\|_{M}^2- \|{\tz}_k-{z_k}\|_{M}^2& \geq\frac{[\sigma-1+(2-\theta-\tau_1)\theta]}{\beta\theta^3}\|\gamma_{k}-\gamma_{k-1}\|^2+\frac{[\sigma-1+(\sigma-\tau_1)\theta]}{\theta}\beta\|B(y_{k}-y_{k-1})\|^2\\
&
 +\frac{2(1-\theta)[\sigma-1+(1-\tau_1)\theta]}{\theta^2}\inner{\gamma_{k-1}-\gamma_{k-2}}{B(y_k-y_{k-1})}\\
 & +\frac{[\sigma-1+(1-\tau_1)\theta]}{\theta}\left(\|y_k-y_{k-1}\|_{H}^2-\|y_{k-1}-y_{k-2}\|_{H}^2\right)
\end{align*}
which, combined with definition of $\{\eta_k\}$ in \eqref{etak}, yields 
\begin{align*}\nonumber
\sigma\|\tz_{k}&-{z_{k-1}}\|_{M}^2- \|{\tz}_k-{z_k}\|_{M}^2+\eta_{k-1}-\eta_{k} 
\geq\frac{[\sigma-1+(\sigma-\tau_1)\theta]\beta}{\theta}\|B(y_{k}-y_{k-1})\|^2\\
&+\frac{[\sigma-1+(2-\theta-\tau_1)\theta]}{\beta\theta^3}\|\gamma_{k-1}-\gamma_{k-2}\|^2+\frac{2(1-\theta)[\sigma-1+(1-\tau_1)\theta]}{\theta^2}\inner{\gamma_{k-1}-\gamma_{k-2}}{B(y_k-y_{k-1})}\\
&=\frac{1}{\theta} \; \left\langle G\left[ 
\begin{array}{c} 
\sqrt{\beta}B(y_k-y_{k-1})\\
 (\gamma_{k-1}-\gamma_{k-2})/\theta\sqrt{\beta}
\end{array} \right],\left[ 
\begin{array}{c} 
\sqrt{\beta}B(y_k-y_{k-1})\\
 (\gamma_{k-1}-\gamma_{k-2})/\theta\sqrt{\beta}
\end{array} \right]\right\rangle 
\end{align*}
where $G$ is as in \eqref{matrixtheta>1}. Therefore, since $G$ is positive semidefinite (see Lemma~\ref{pro:sigma}(b)), we conclude that  inequality \eqref{breg-cond1}  also holds for $k>1$.
To end the proof,  note that the last statement  of the proposition follows  trivially from the first one and  Lemma~ \ref{lem:inclusion}.
\qed
\end{proof}

We are now ready to present our main results of this paper, namely, we establish pointwise and ergodic 
 iteration-complexity bounds  for the partially inexact proximal ADMM.
%%%%%%
\begin{theorem} \label{th:pointwise} 
Consider the sequences $\{(x_k,y_k,\gamma_k)\}$ and $\{(\tilde x_k,\tilde \gamma_k)\}$  generated by the partially inexact  proximal ADMM.
Then,  for every $k\geq 1$, 
\begin{equation}\label{eq:th_incADMMtheta1} 
\left( 
\begin{array}{c} 
\frac{1}{\beta}(x_{k-1}-x_k)\\[1mm]  
(H+\beta B^*B)(y_{k-1}-y_k)\\[1mm]  
\frac{1}{\beta\theta}(\gamma_{k-1}-\gamma_k)
\end{array} 
\right) \in 
\left[ 
\begin{array}{c} 
\partial f(\tilde{x}_k)- A^*\tilde{\gamma}_k\\[1mm]  
\partial g(y_k)- B^*\tilde{\gamma}_k\\[1mm]  
A\tilde{x}_k+By_k-b
\end{array} \right]
\end{equation}
and   there exist  $\sigma\in ]0,1[$ and  $i\leq k$ such that
\[
\left( \frac{1}{\beta} \|x_{i-1}-x_i\|^2+\|y_{i-1}-y_i\|_{(H+\beta B^*B)}^2+\frac{1}{\beta\theta}\norm{\gamma_{i-1}-\gamma_i}^2\right)^{1/2}\leq \frac{\sqrt{d_0}}{\sqrt{k}} \sqrt{\frac{2(1+\sigma)+4\mu }{1-\sigma}}
\]
where  $d_0$  and $\mu$ are  as in  \eqref{def:d0admm1} and \eqref{def:eta0}, respectively. 
\end{theorem}
\begin{proof} This result follows by combining
Proposition~\ref{maincorADMM1} and Theorem~\ref{th:pointwiseHPE}.
\qed
\end{proof}
\begin{remark}\label{rem:admmpoint} 
For a given tolerance $\bar \rho>0$, Theorem~\ref{th:pointwise}  ensures  that in at most $\mathcal{O}(1/\bar\rho^2)$ iterations, the partially inexact proximal ADMM  provides an approximate solution $\tilde z:=(\tilde x, y, \tilde \gamma)$  of \eqref{FAB} together with a residual $r:=(r_x,r_y,r_{\gamma})$ in the sense that
\[
\frac{1}{\beta}r_x\in\partial f(\tilde{x})- A^*\tilde{\gamma}, \qquad (H+\beta B^*B)r_y\in \partial g(y)- B^*\tilde{\gamma}, \qquad  \frac{1}{\beta\theta}r_{\gamma}= A\tilde{x}+By-b, \quad \|(r_x,r_y,r_\gamma)\|_M\leq \bar \rho,
\]
where  $M$ is as in \eqref{def:matrixM1}. 
Note that the above relations are equivalent to $M r\in T(\tilde z)$ and $\|r\|_M\leq \bar \rho$ with $T$  as in \eqref{FAB}. 
%See Remark~\ref{Rem:er} for a discussion about the latter approximate solution concept.
%As already mentioned in Remark~\ref{Rem:er}, if $\|r\|_M=0$, then   $Mr=0$ or, equivalently,  $\tilde z$ is a solution of  \eqref{FAB}.
%(see Remark~\ref{Rem:er})
%Note that $(x,y)$  can be seen as an approximate solution of  problem~\eqref{optl} and $\gamma$ is an approximate Lagrangian multiplier.
\end{remark}

\begin{theorem}\label{th:ergodicproximal ADMM}
Let the sequences $\{(x_k,y_k,\gamma_k)\}$ and $\{(\tilde x_k,\tilde \gamma_k)\}$ be  generated by the partially inexact  proximal ADMM.
Consider the ergodic sequences
$\{(x^a_k,y^a_k,\gamma^a_k)\}$, $\{(\tilde x^a_k,\tilde\gamma^a_k)\}$, $\{( r^{a}_{k,x},r^{a}_{k,y},r^{a}_{k,\gamma})\}$ and 
$\{(\varepsilon_{k,x}^a,\varepsilon_{k,y}^a)\}$ defined by
\begin{equation}\label{eq:jase12}
 (x_k^a,y_k^a,\gamma_k^a)=\frac1k\sum_{i=1}^k\left( x_i, y_i,\gamma_i  \right), \qquad
  (\tilde x_k^a,\tilde\gamma_k^a)=\frac1k\sum_{i=1}^k\left( \tilde{x}_i,\tilde\gamma_i \right),\qquad
  ( r^{a}_{k,x},r^{a}_{k,y},r^{a}_{k,\gamma})=\frac{1}{k}\sum_{i=1}^k\left(r_{i,x},r_{i,y},r_{i,\gamma}  \right),
\end{equation}
\begin{equation}\label{eq:at09}
(\varepsilon^a_{k,x},\varepsilon^a_{k,y})= \frac{1}{k}\sum_{i=1}^k\left(\Inner{r_{i,x}/\beta+A^{*}\tilde\gamma_i}{\tilde{x}_i-\tilde{x}_k^a}, \Inner{\left(H+\beta B^*B\right)r_{i,y}+B^{*}\tilde\gamma_i}{y_i-y_k^a}\right)
\end{equation}
where 
\begin{equation}\label{eq:291} 
(r_{i,x},r_{i,y},r_{i,\gamma})=\left(x_{i-1}-x_{i}, y_{i-1}-y_{i}, \gamma_{i-1}-\gamma_{i}\right).
\end{equation}
Then, for every $k\geq 1$, we have $\varepsilon^a_{k,x}, \varepsilon^a_{k,y} \geq 0$, 
\begin{equation}\label{eq:th_incADMMtheta<1} 
\left( 
\begin{array}{c} 
\frac{1}{\beta} r_{k,x}^{a}\\[1mm]  
(H+\beta B^*B)r_{k,y}^{a}\\[1mm]  
\frac{1}{\beta\theta}r_{k,\gamma}^{a}
\end{array} 
\right) \in 
\left[ 
\begin{array}{c} 
\partial_{\varepsilon^a_{k,x}} f(x_k^a)- A^*\tilde{\gamma}_k^a\\[1mm]  
\partial_{\varepsilon^a_{k,y}} g(y_k^a)- B^*\tilde{\gamma}_k^a\\[1mm]  
Ax_k^a+By_k^a-b,
\end{array} \right],
\end{equation}  
and  there exists  $\sigma\in ]0,1[$  such that
\begin{equation}\label{eq:ad45}
  \left(\frac{1}{\beta}\|r_{k,x}^{a}\|^2+\|r_{k,y}^a\|_{(H+\beta B^*B)}^2+\frac{1}{\beta\theta}\norm{r_{k,\gamma}^a}^2\right)^{1/2}\le 
  \frac{2\sqrt{(1+\mu) d_0}}{k}
    \end{equation}
 and
\begin{equation}\label{eq:ad451}
 \varepsilon^a_{k,x} + \varepsilon^a_{k,y} \leq \frac{ 3(1+\mu)(3-2\sigma) d_{0} }{2(1-\sigma)k}
    \end{equation}
where  $d_0$  and $\mu$ are  as in  \eqref{def:d0admm1} and \eqref{def:eta0}, respectively.  
\end{theorem}
\begin{proof}
By  combining  Proposition~\ref{maincorADMM1} and Theorem~\ref{th:ergHPE}, 
we conclude that inequality \eqref{eq:ad45} holds, and
\begin{equation}\label{ine:eps:pr}
\varepsilon_k^a \leq \frac{ 3(1+\mu)(3-2\sigma) d_{0} }{2(1-\sigma)k},
\end{equation} 
where 
\begin{equation}\label{ine:eps124523}
\varepsilon_k^a=\dfrac{1}{k}\left(\sum_{i=1}^k\,\Inner{r_{i,x}/\beta}{\tilde{x}_i-\tilde{x}_k^a}+
\sum_{i=1}^k\,\Inner{\left(H+\beta B^*B\right)r_{i,y}}{y_i-y_k^a}+
\sum_{i=1}^k\,\Inner{r_{i,\gamma}/(\theta\beta)}{\tilde \gamma_i-\tilde \gamma_k^a}\right)
\end{equation}
%where $(p_{i},q_{i},s_{i})=\left(\frac{1}{\beta}r_{i,x},\left(H+\beta B^*B\right)r_{i,y},\frac{1}{\theta\beta}r_{i,\gamma}\right)$.
On the other hand, \eqref{aux.11},  \eqref{eq:jase12} and \eqref{eq:291}  yield
\begin{align*}
 A\tilde{x}_k+By_k=\frac{1}{\theta\beta}r_{k,\gamma}+b,\quad  A\tilde{x}^a_k+By^a_k=\frac{1}{\theta\beta}r^a_{k,\gamma}+b. 
 \end{align*}
%where $r^a_{k,\gamma}=\sum_{i=1}^k(\gamma_{i-1}-\gamma_{i})/(\theta\beta)$.
Additionally,   it follows from definitions of $r_{i,\gamma}$ and $r^a_{k,\gamma}$ that 
\begin{align*}
 \frac{1}{k} \sum_{i=1}^k\inner{\tilde \gamma_i}{r_{i,\gamma}-r_{k,\gamma}^a}=
 \frac{1}{k} \sum_{i=1}^k\inner{\tilde \gamma_i-\tilde \gamma_k^a}{r_{i,\gamma}-r_{k,\gamma}^a}=
 \frac{1}{k} \sum_{i=1}^k\inner{\tilde \gamma_i-\tilde \gamma_k^a}{r_{i,\gamma}}.
 \end{align*}
Hence, combining  the identity in \eqref{ine:eps124523} with the last two  equations, we have
\begin{align}\nonumber
\varepsilon_k^a &=\dfrac{1}{k}\sum_{i=1}^k\,\Big(\Inner{r_{i,x}/\beta}{\tilde{x}_i-\tilde{x}_k^a}+
\Inner{\left(H+\beta B^*B\right)r_{i,y}}{y_i-y_k^a}\Big)+
\dfrac{1}{k}\sum_{i=1}^k\Inner{\tilde \gamma_i}{\left(r_{i,\gamma}-r^a_{k,\gamma}\right)/\left(\theta\beta\right)}\\\nonumber
 &=\dfrac{1}{k}\sum_{i=1}^k\,\Big(\inner{r_{i,x}/\beta}{\tilde{x}_i-\tilde{x}_k^a}+
\Inner{\left(H+\beta B^*B\right)r_{i,y}}{y_i-y_k^a}+\Inner{\tilde \gamma_i}{A\tilde{x}_i-A\tilde{x}_k^a+By_i-By_k^a}\Big)\\\nonumber
&=\frac{1}{k}\sum_{i=1}^k \Inner{r_{i,x}/\beta+A^*\tilde \gamma_i}{\tilde{x}_i-\tilde{x}_k^a}+
\frac{1}{{k}}\sum_{i=1}^k \Inner{\left(H+\beta B^*B\right)r_{i,y}+B^*\tilde \gamma_i}{y_i-y_k^a} =\varepsilon_{k,x}^a+\varepsilon_{k,y}^a,
\end{align}
where the last equality is due to the definitions of $\varepsilon_{k,x}^a$ and $\varepsilon_{k,y}^a$  in \eqref{eq:at09}.
Therefore, the inequality in \eqref{eq:ad451} follows trivially from   the last equality and  \eqref{ine:eps:pr}. 

To finish the proof 
of the theorem, note that direct use of Theorem \ref{for:trans}(b) (for $f$ and $g$), \eqref{eq:th_incADMMtheta1}--\eqref{eq:291} 
give $\varepsilon_{k,x}^a,\,\varepsilon_{k,y}^a\geq 0$ and the inclusion in \eqref{eq:th_incADMMtheta<1}.
\qed
\end{proof}
\begin{remark} 
For a given tolerance $\bar \rho>0$, Theorem~\ref{th:ergodicproximal ADMM} ensures  that in at most $\mathcal{O}(1/\bar\rho)$ iterations, the partially inexact proximal ADMM  provides, in the ergodic sense, an approximate solution $\tilde z:=(\tilde x^a, y^a, \tilde \gamma^a)$  of \eqref{FAB} together with  residues $r:=(r_x^a,r_y^a,r_{\gamma}^a)$ and $(\varepsilon_{x}^a, \varepsilon_{y}^a)$ such that 
\[
\frac{1}{\beta}r_x^a\in\partial_{\varepsilon_{x}^a} f(\tilde{x}^a)- A^*\tilde{\gamma}^a, \;\; (H+\beta B^*B)r_y^a\in \partial_{\varepsilon_{y}^a} g(y^a)- B^*\tilde{\gamma}^a, \;\;  \frac{1}{\beta\theta}r_{\gamma}^a= A\tilde{x}^a+By^a-b, \quad \|(r_x^a,r_y^a,r_{\gamma}^a)\|_M\leq \bar \rho,
\]
where  $M$ is as in \eqref{def:matrixM1}.    The above ergodic complexity bound  is better than  the pointwise one   by  a factor of $\mathcal{O}(1/\bar \rho)$; however, the  above inclusion  is, in general,  weaker than that of the pointwise case due to   the $\varepsilon$-subdifferentials of the $f$ and $g$ instead of the subdifferentials. 
%Second, as mentioned before, the sequence $\{\rho_k\}$  as in \eqref{}.... with ... is bounded even though $\sigma$ can be equal to one and no boundedness assumption is assumed on the operator $T$may be c\ an instance of the modified HPE framework with the latter property.
\end{remark}
%%%%%%%%%%%%%%%%%%%%%%%%

%%%%%%%%%%%%%%%%%%%%%%%%

\section{Numerical Experiments} \label{sec:exp}

In this section, we report some numerical  experiments to illustrate the performance of the partially inexact proximal ADMM (PIP-ADMM) on two classes of problems, namely, LASSO  and $L _1-$regularized logistic  regression. 
Our  main goal  is to show that, in some applications,  the method performs better with a stepsize parameter $\theta > 1$ instead of  the choice $\theta=1$ as considered in the related literature.
Similarly to  \cite{Eckstein2017App,Eckstein2017Relat},  we also used a hybrid inner stopping criterion for the PIP-ADMM, i.e.,  the inner-loop terminates  when $v_k$ satisfies either the inequality in~\eqref{cond:inex} or  $\|v_k\|\leq 10^{-8}$.   This strategy is motivated by the fact that, close to approximate solutions, the former condition seems to be more restrictive than the latter. We set $\tau_1= 0.99(1+\theta-\theta^2)/(\theta(2-\theta))$,  $\tau_2=1-10^{-8}$
and $H=0$.
For a comparison purpose, we also run  \cite[Algorithm~2]{Eckstein2017App}, denoted here by  relerr-ADMM; see  Remark~\ref{remark23}(iv) for more details on the relationship between the PIP-ADMM and the  relerr-ADMM.
As suggested by \cite{Eckstein2017App},  the  error tolerance parameter $\tau_1$ in~\eqref{er:EC} was taken equal to   $0.99$.
 For all tests, both algorithms used the initial point  $(x_0,y_0,\gamma_0)=(0,0,0)$, the penalty parameter $\beta=1$, and stopped when the following condition was satisfied
\[
\|(x_k-x_{k-1},y_k-y_{k-1},\gamma_k-\gamma_{k-1})\|_{M}\leq 10^{-2},
\]
where $M$ is as in \eqref{def:matrixM1}.  
The computational results were obtained using MATLAB R2015a on
 a 2.4GHz Intel(R) Core~i7  computer with 8GB of RAM.% and Windows 10 home?????????????? single language. 

 \subsection{LASSO Problem} \label{sub_lasso}
We consider to approximately solve the LASSO  problem \cite{10.2307/2346178,tibshirani2013}
  \begin{equation*} 
\min_{x \in \R^n}  \frac{1}{2}\|Cx-d\|^2+\delta\|x\|_1 
\end{equation*}
 where $C \in \R ^{m \times n}$, $d\in \R ^m$, and  $\delta$ is  a regularization
parameter. We set $\delta=0.1\|C^*d\|_{\infty}$.  By introducing a new variable,  we can rewrite the above problem as
  \begin{equation} \label{optl245}
\min \left\{ \frac{1}{2}\|Cx-d\|^2+\delta\|y\|_1: \; y-x=0,  \;x\in \R^n, y\in \R ^{n} \right\}.
\end{equation}
Obviously,  \eqref{optl245} is an instance of \eqref{optl} with $f(x)=({1}/{2})\|Cx-d\|^2$, $g(y)=\delta\|y\|_1$, $A=-I$, $B=I$ and $b=0$. Note that, in this case, the pair $(\tilde x_k,\tilde v_k)$ in 
\eqref{cond:inex} can be obtained by computing  an approximate solution  $\tilde x_k$ with a residual $\tilde v_k$  of the following  linear system
 \[
 (C^*C+\beta I) x=  (C^*d+\beta y_{k-1}-\gamma_{k-1}).
 \]
For approximately solving the above  linear system, we used the conjugate gradient method \cite{nocedal2006numerical} with starting point $C^*d+\beta y_{k-1}-\gamma_{k-1}$. Note also that subproblem  \eqref{def:tyk-admm1} has a closed-form solution
 \[
 y_k=\mbox{shrinkage}_{\delta/\beta}\left({\tilde{x}_k +\gamma_{k-1}/\beta}\right),
 \]
 where the shrinkage operator is defined as 
\begin{equation}\label{shrin}
\mbox{shrinkage}_{\kappa}:\R^n\to\R^n, \quad (\mbox{shrinkage}_{\kappa}(a))_i=\mbox{sign}(a_i)\max(0,|a_i|-\kappa) \quad i=1,2,\ldots,n,
\end{equation}
with   sign$(\cdot)$  denoting  the sign function.

We first tested the  methods for solving 3 randomly generated LASSO problem instances.   
%with vector $d$ and matrix $C$  randomly generated as follows: 
  For a given dimension $m \times n$, we generated a random   matrix $C$  and   scaled its  columns to have unit $l_2$-norm. The vector $d\in \R^m$ was chosen as $d = Cx + \sqrt{0.001}y$,  where the $(100/n)-$sparse  vector  $x\in \R^n$   and  the noisy vector $y\in \R^m$ were also generated randomly.

%
%\begin{lstlisting}
%x = sprandn(n,1,100/n);
%C = randn(m,n);
%C = C*spdiags(1./sqrt(sum(C.^2))',0,n,n); 
%d = C*x + sqrt(0.001)*randn(m,1);
%\end{lstlisting}

%Para definir o local da página em que a tabela ficará situada pode-se colocar:
%h- Ficará onde foi digitado;
%b- Ficará na parte inferior da página;
%t- Ficará na parte superior da página;
%p- Ficará em página separada.
\begin{table}[h]\caption{Performance of the relerr-ADMM and PIP-ADMM to solve 3 randomly generated LASSO problems.}
\vspace{.1 in}
\resizebox{\textwidth}{!}{ % abre resizebox, setar tabela da largura da página.
%\begin{center}
\begin{tabular}{|c|ccc|ccc|ccc|ccc|}  \hline  %\toprule 
{Dim. of $A$} & \multicolumn{3}{|c|}{relerr-ADMM}&\multicolumn{3}{|c|}{PIP-ADMM ($\theta =1$)} &\multicolumn{3}{|c|}{PIP-ADMM ($\theta =1.3$)} &
\multicolumn{3}{|c|}{PIP-ADMM ($\theta =1.6$)} \\ \cline{2-13}
$m\times n$        &Out  &Inner &Time &Out  &Inner  &Time  &Out  &Inner  &Time  &Out  &Inner  &Time \\\hline     %\midrule
$900 \times 3000$  &26   &195   &11.1 &26   &195    &10.2  &22   &169    &8.8   &19   &172    &7.9  \\
$1200\times 4000$  &26   &193   &22.7 &26   &193    &20.9  &21   &155    &20.9  &19   &169    &17.9 \\ 
$1500\times 5000$  &25   &185   &40.9 &25   &185    &36.7  &21   &158    &34.0  &18   &159    &29.3 \\ \hline %\bottomrule
%                   &    & 622    &195    &195    &169    &172 \\
%                   &    &13.08  &11.05  &10.25  &8.79   &7.93 \\ \hline
%$1200\times 4000$  &    & 26    & 26    &26     &21     &19 \\  
%                   &    & 622    &193    &193    &155    &169 \\
%                   &    &25.66  &22.71  &20.93  &20.93   &17.94 \\ \hline
%$1500\times 5000$  &   & 25     & 25    &25     &21      &18 \\  
%                   &   & 598    &185    &185    &158     &159 \\
%                   &  &43.70   &40.94  &36.75  &33.96   &29.33 \\ \bottomrule                
\end{tabular}}
%\end{center}
\end{table}

We also tested the methods on five standard data sets from the Elvira  biomedical data set repository \cite{ELVIRA}. 
The first data set is  the colon tumor gene expression \cite{alon1999broad} with $m=62$ and $n=2000$, the second is the central nervous system (CNS) data \cite{pomeroy2002prediction} with $m=60$ and $n=7129$,  the third is the  prostate cancer data \cite{singh2002gene} with $m=102$ and $n=12600$,  the fourth is  the Leukemia cancer-ALLMLL data  \cite{Golub531} with $m=38$ and $n=7129$, and the fifth is the lung cancer-Michigan data \cite{beer2002gene}  with $m=96$ and $n=7129$. As in the randomly generated problems, we scaled the columns of  $C$ in order to have unit $l_2$-norm.
\begin{table}[h]\caption{Performance of the relerr-ADMM and PIP-ADMM on 5 data sets.}
\vspace{.1 in}
\resizebox{\textwidth}{!}{ % abre resizebox, setar tabela da largura da página.
\begin{tabular}{|l|ccc|ccc|ccc|ccc|}  \hline  %\toprule 
\multicolumn{1}{|c|}{\multirow{2}{*}{Data set}}& \multicolumn{3}{c|}{ relerr-ADMM}&\multicolumn{3}{c|}{PIP-ADMM ($\theta =1$)} 
&\multicolumn{3}{c|}{PIP-ADMM ($\theta =1.3$)} &\multicolumn{3}{c|}{PIP-ADMM ($\theta =1.6$)} \\ \cline{2-13}
                       &Out  &Inner &Time &Out  &Inner  &Time  &Out  &Inner  &Time  &Out  &Inner  &Time \\\hline     %\midrule
                       Colon &87 &1535 &11.9 &87 &1517 &11.9 &78 &1378 &10.8 &72 &1390 &10.2 \\
CNS &204 &5979 &466.6 &204 &5967 &467.1 &179 &5293 &425.7 &164 &5267 &383.5 \\  %Central Nervous System
Prostate &368 &16176 &3523.5 &366 &16030 &3502.6 &298 &13212 &2791.2 &252 &12319 &2642.4 \\ 
%Colon &330 &8717 &55.3 &352 &7876 &53.4 &359 &8170 &56.2 &352 &8494 &62.6 \\ 
Leukemia &415 &7435 &813.3 &415 &7435 &811.6 &347 &6290 &674.2 &297 &5710 &591.4 \\
Lung  &485 &10975 &1008.6 &485 &10949 &1023.4  &379 &8612 &805.6 &314 &7736 &679.1 \\ \hline %\bottomrule             
\end{tabular}}
\end{table}

The performances of the relerr-ADMM and PIP-ADMM are listed in Tables  1 and 2, in which ``Out" and ``Inner" denote the number of iterations and the total number of inner iterations  of the methods, respectively,  whereas ``Time" is the CPU time in seconds. From these tables, we see that  the relerr-ADMM and the PIP-ADMM with $\theta=1$ had similar performances. However,  the PIP-ADMM with $\theta=1.3$ and $\theta=1.6$ clearly outperformed    the relerr-ADMM.

%Therefore, 

%Therefore, the advantages of considering a stepsize parameter $\theta>1$ in the PIP-ADMM is verified.

%Therefore, we  conclude that the choice of a stepsize parameter $\theta>1$  substantially improved  the numerical  performance of the PIP-ADMM for the %tested LASSO problems.

\subsection{$L _1-$regularized Logistic Regression }

%We  tested the DR-ADMM to solve the total variation (TV) denoising problem  \cite{???} 

Consider the  $L _1-$regularized logistic regression problem  \cite{koh2007interior} 
 \begin{equation*} 
\min_ {(u,t) \in \R^n \times \R}\left\{  \sum_{i=1}^{m} \log\left(1+\exp\left(-d_i[\inner{c_i}{u}+t]\right)\right)+\delta m\Norm{u}_1 \right\},
\end{equation*}
 where $(c_i,d_i)\in \R^n\times \{-1,+1\}$, for every $i=1,\dots,m$, and $\delta$ is  a regularization parameter. We set  $\delta = 0.5\lambda_{\mbox{max}}$, where  $\lambda_{\mbox{max}}$ is  defined as in \cite[Subsection 2.1]{koh2007interior}.
 Note that the above problem can be rewritten as 
\begin{equation} \label{optl23}
\min_{(x,u,t)\in \R^{n+1}\times\R^n\times\R}\left\{ \sum_{i=1}^{m}\log\left(1+\exp\left(-d_i\Inner{(1,c_i)}{x}\right)\right)+
\delta m \Norm{u}_1\, :\; (u,t)-x=0\right\},
\end{equation}
%it is easy to see that 3 is a particular case of
which is an instance of \eqref{optl} with $f(x)=\sum_{i=1}^{m}\log\left(1+\exp\left(-d_i\Inner{(1,c_i)}{x}\right)\right)$, 
$g(y)=g(u,t)= m\delta\|u\|_1$, $A=-I$, $B=I$, and $b=0$.   
 In this case, the pair $(\tilde x_k,\tilde v_k)$ in 
\eqref{cond:inex} was obtained as follows:
 the iterate $\tilde x_k$ was computed by the Newton method \cite{nocedal2006numerical} with starting point equal to $(0,\ldots,0)$,    as an approximate solution of the following unconstrained optimization problem 
 \[
\min_{x \in \R^{n+1}} \left \{h(x)=\sum_{i=1}^{m}\log\left(1+\exp\left(-d_i\Inner{(1,c_i)}{x}\right)\right) + \inner{x} {{\gamma}_{k-1}}+
\frac{\beta}{2} \|  y_{k-1} -x \|^2\right\}, 
 \]
 whereas $\tilde v_k$  was taken as  $\tilde v_k=\nabla h(\tilde x_k)$. 
 Note that \eqref{def:tyk-admm1}  has a closed-form solution $y_k=( u_k,t_k)$ given by
 \[
 u_k=\mbox{shrinkage}_{m\delta/\beta}\left({\tilde{x}^u_k +\gamma^u_{k-1}/\beta}\right), \quad t_k=\tilde{x}^t_k+\gamma^t_{k-1}/\beta,
 \]
where $\tilde{x}^u_k, \gamma^u_k \in \R^n$ and $\tilde{x}^t_k, \gamma^t_k \in \R^t$ are the components of the vectors $\tilde{x}_k$ and $\gamma_k$, i.e.,  $(\tilde{x}^u_k,\tilde{x}^t_k)=\tilde{x}_k$ and  $(\gamma^u_k,\gamma^t_k)=\gamma_k$, and the operator  shrinkage is  as  in   \eqref{shrin}.

We tested the methods for solving seven $L _1-$ regularized logistic regression problem instances. We selected four instances of  Section~\ref{sub_lasso},
and  three  from  the ICU Machine Learning Repository \cite{Dua:2017}, namely, the ionosphere data \cite{sigillito1989classification} with $m=351$ and $n=34$, the secom data with $m=1567$ and $n=590$, and the spambase data with $m=4601$ and $n=57$.
We also scaled  the columns (resp. rows)  of  $C=[c_1,\ldots,c_n]^*$ to have unit $l_2$-norm when $n\geq m$ (resp.  $m>n$).

Tables~3  reports the performances of  the relerr-ADMM and PIP-ADMM for solving the aforementioned seven instances of the problem \eqref{optl23}.  In Table  3,  ``Out" and ``Inner" are the number of  iterations and the total of inner iterations of the methods, respectively,  whereas ``Time" is the CPU time in seconds. 
Similarly  to the numerical results of Section~\ref{sub_lasso}, we observe that the relerr-ADMM and the PIP-ADMM with $\theta=1$ had similar performances, whereas the
PIP-ADMM with $\theta=1.3$ and $\theta=1.6$  outperformed   the relerr-ADMM.
Therefore, the efficiency of the PIP-ADMM for solving real-life applications is illustrated.

%From this table, we see that  the relerr-ADMM and the PIP-ADMM with $\theta=1$ had similar behaviour. However,  the PIP-ADMM with $\theta=1.3$ %and $\theta=1.6$, clearly outperformed the latter  two methods.

\begin{table}[h]\caption{ Performance of the relerr-ADMM and PIP-ADMM on 7 data sets.}
\vspace{.1 in}
%\def\arraystretch{1.5} % esp. entre linhas
%\footnotesize{
\resizebox{\textwidth}{!}{ % abre resizebox, setar tabela da largura da página.
\begin{tabular}{|l|ccc|ccc|ccc|ccc|}  \hline  %\toprule 
\multicolumn{1}{|c|}{\multirow{2}{*}{Data set}}& \multicolumn{3}{c|}{relerr-ADMM}&\multicolumn{3}{c|}{PIP-ADMM ($\theta =1$)} 
&\multicolumn{3}{c|}{PIP-ADMM ($\theta =1.3$)} &\multicolumn{3}{c|}{PIP-ADMM ($\theta =1.6$)} \\ \cline{2-13}
                       &Out  &Inner &Time &Out  &Inner  &Time  &Out  &Inner  &Time  &Out  &Inner  &Time \\\hline     %\midrule
CNS &153 &753 &6545.3 &153 &753 &6797.9 &128 &630 &6298.5 &113 &564 &5357.8 \\  %Central Nervous System
Colon &149 &596 &172.2 &149 &596 &180.5 &125 &500 &150.5 &110 &464 &139.0 \\
Leukemia &139 &693 &6264.4 &139 &693 &6248.8 &120 &592 &5203.9 &112 &563 &4951.9 \\
Lung &225 &1333 &11676.9 &225 &1333 &11354.4 &219 &1304 &10910.7 &215 &1321 &11152.5 \\
Ionosphere &54 &208 &0.2 &54 &208 &0.2 &42 &162 &0.2 &35 &142 &0.1 \\ 
Secom &21 &122 &15.0 &21 &121 &15.0 &17 &97 &13.5 &15 &89 &12.4 \\ 
Spambase &47 &212 &29.7 &47 &212 &29.8 &37 &168 &25.7 &30 &147 &22.4 \\ \hline %\bottomrule            
\end{tabular}}
\end{table}

\section{Conclusions}

In this paper, we proposed a partially inexact proximal ADMM  and established  pointwise and ergodic iteration-complexity bounds for it. 
The proposed  method allows its first subproblem  to be solved inexactly using a relative approximate criterion, whereas a stepsize parameter is added  in the updating rule of the Lagrangian multiplier in order to  improve its computational  performance.
%One of the advantages of the method is the stepsize parameter added to the updating rule of the Lagrangian multiplier. 
We presented some computational results illustrating  the numerical advantages of  the method.

\bibliographystyle{spmpsci_unsrt}
%\bibliographystyle{amsplain}
%\bibliography{VM_ADMM_ref}

\def\cprime{$'$}

\end{document}